\documentclass[12pt]{article}
\usepackage{amsmath}
\usepackage{amsfonts}
\usepackage{amssymb}
\usepackage{xcolor}
\usepackage{bigints}
\usepackage{relsize}
\newtheorem{teor}{Theorem}[section]
\newtheorem{prop}[teor]{Proposition}
\newtheorem{coro}[teor]{Corollary}
\newtheorem{lema}[teor]{Lemma}
\newtheorem{defi}{Definition}[section]
\newtheorem{ejem}{Example}[section]
\newtheorem{nota}{Remark}[section]

\allowdisplaybreaks %Para permitir que las ecuaciones sigan en p\'aginas distintas

\newenvironment{proof}[1][Proof]{\textbf{#1.} }{\ \rule{0.5em}{0.5em}}

\begin{document}
\def\N{\mathbb{N}}
\def\Z{\mathbb{Z}}
\def\K{\mathbb{K}}
\def\R{\mathbb{R}}
\def\D{\mathbb{D}}
\def\C{\mathbb{C}}
\def\T{\mathbb{T}}

\def\dint{\int}
\def\intc{\int_0^1}
\def\intD{\int_{\D}}
\def\intpi{\int_0^{2\pi}}
\def\sumi{\sum_{n=0}^\infty}
\def\dfrac{\frac}
\def\dsum{\sum}
\def\cuadro{\hfill $\Box$}
\def\qed{\hfill $\Box$}
\def\prueba{\vskip10pt\noindent{\it PROOF.}\hskip10pt}
\newcommand{\cuadrosymb}{\mbox{ }~\hfill~{\rule{2mm}{2mm}}}
%%%%%%%%%%%%%%%%%%%%%%%%%%%%%%%%%%%%%%%%%%%%%%%%%%%%%%%%%%%%%%%%
\providecommand{\norm}[1]{\lVert#1\rVert}
\providecommand{\Norm}[1]{\left\lVert#1\right\rVert}
\providecommand{\presc}[2]{\langle #1 ,#2\rangle}
\providecommand{\Presc}[2]{\left\langle #1 ,#2\right\rangle}
\providecommand{\lcn}[2]{\mathcal{L}(#1,#2)}
\providecommand{\convsot}[1]{\xrightarrow[]{SOT}#1}
\providecommand{\convwot}[1]{\xrightarrow[]{WOT}#1}
\providecommand{\normop}[1]{\lVert#1\rVert_{{op_2}}}
\providecommand{\normb}[1]{\lVert#1\rVert_{B(\ell^2)}}
\providecommand{\normm}[1]{\lVert#1\rVert_{M(\ell^2)}}
\providecommand{\normu}[1]{\lVert#1\rVert_{U(\ell^2)}}
\providecommand{\normc}[1]{\lVert#1\rVert_{C(\ell^2)}}
\providecommand{\normbv}[1]{\lVert#1\rVert_{BV(\ell^2)}}
\providecommand{\normel}[1]{\lVert#1\rVert_{E(\ell^2)}}
\providecommand{\normcr}[1]{\lVert#1\rVert_{C_r(\ell^2)}}
\providecommand{\normbh}[1]{\lVert#1\rVert_{\Bl2}}
\providecommand{\normmh}[1]{\lVert#1\rVert_{\Ml2}}
\providecommand{\normbx}[1]{\lVert#1\rVert_{B(\ell^2(X))}}
\providecommand{\normmx}[1]{\lVert#1\rVert_{M(\ell^2(X))}}
\providecommand{\normms}[1]{\lVert#1\rVert_{ms}}
\providecommand{\Norm}[1]{\left\lVert#1\right\rVert}
\providecommand{\lk}[1]{\mathcal{L}_k^{#1}}

\providecommand{\htild}[1]{\tilde{H}^2(\mathbb{T},\mathcal{L}(H))}
\providecommand{\ltwolh}[1]{\ell^2(\mathcal{L}(H))}
\providecommand{\ltwoh}[1]{\ell^2(H)}
\providecommand{\ltwosot}[1]{\ell^2_{SOT}(\mathcal{L}(H))}
\providecommand{\linftylh}[1]{\ell^\infty(\mathcal{L}(H))}

\providecommand{\bltwoh}[1]{\mathcal{B}\left(\ell^2(H)\right)}
%%%%%%%%%%%%%%%%%%%%%%%%%%%%%%%%%%%%%%%%%%%%%%%%%%%%%%%%%%%%%%%%%%%%%%%
\def\L{{\mathcal L}}
\def\P{{\cal P}}
\def\B{\mathcal B}
\def\KO{{\cal K}}
\def\J{\mathcal J}

\def\M{\mathcal M}

\def\botimes{{\bf \otimes}}
\def\ba{\begin{eqnarray*}}
\def\ea{\end{eqnarray*}}

\def\be{\begin{equation}}

\def\ee{\end{equation}}

\def\A{{\bf A}}
\def\bB{{\bf B}}
\def\bT{{\bf T}}
\def\x{{\bf x}}
\def\y{{\bf y}}
\def\z{{\bf z}}

\def\dt{\frac{dt}{2\pi}}
\def\ds{\frac{ds}{2\pi}}

\def\d{\displaystyle}

\def\Tkj{T_{kj}}
\def\Skj{S_{kj}}
\def\Rk{{\bf R}_k}
\def\Dl{{\bf D}_l}
\def\Cj{{\bf C}_j}
\def\ej{{\bf e}_j}
\def\ek{{\bf e}_k}

\def\la{\langle}
\def\ra{\rangle}
\def\e{\varepsilon}

\def\Ll2{\mathcal L^1(\ell^2(H))}
\def\Lrl2{\mathcal L^1_r(\ell^2(H))}
\def\Lll2{\mathcal L^1_l(\ell^2(H))}
\def\Cl2{\mathcal C(\ell^2(H))}
\def\Pl2{\mathcal P(\ell^2(H))}
\def\Bl2{\mathcal B(\ell^2(H))}
\def\Ml2{\mathcal M_r(\ell^2(H))}
\def\Multleft{\mathcal M_l(\ell^2(H))}
\def\Multright{\mathcal M_r(\ell^2(H))}
\def\Mult{\mathcal M(\ell^2(H))}
\def\L1l2{L^1(\ell^2(H))}
\def\l2{\ell^2(H)}

\def\sj{\sum_{j=1}^\infty}
\def\sk{\sum_{k=1}^\infty}

\def\fk{\varphi_k}
\def\fj{\varphi_j}

%%%%Spaces

\def\lsot{\ell^2_{SOT}(\N,\B(H))}
\def\lsott{\ell^2_{SOT}(\N^2,\B(H))}

\title{A class of Schur multipliers of matrices with operator entries}
\author{Oscar Blasco, Ismael Garc\'{\i}a-Bayona\thanks{%
Partially supported by  MTM2014-53009-P(MINECO Spain) and  FPU14/01032 (MECD Spain)}}
\date{}

\maketitle

\begin{abstract} In this paper, we will consider matrices with entries in the space of operators $\mathcal{B}(H)$, where $H$ is a separable Hilbert space, and consider the class of (left or right) Schur multipliers  that can be approached in the multiplier norm by matrices with a finite number of diagonals. We will concentrate on the case of Toeplitz matrices and of upper triangular matrices to get some connections with spaces of vector-valued functions.
\end{abstract}

%AMS Subj. Class: Primary 46E40, Secondary 47A56; 15B05
AMS Subj. Class: Primary 47L10; 46E40, Secondary 47A56; 15B05; 46G10.


\section{Introduction. }

Recall that a bounded operator acting on the Hilbert space $\ell^2$, say $T\in\B(\ell^2)$, can be identified with a matrix $A=(\alpha_{kj})$ whose entries are given by $\alpha_{kj}=\la T(e_j),e_k\ra$ where $(e_j)$ stands for the standard orthonormal basis of $\ell^2$, and we use the notation $(Ax)_k= \sum_{j=1}^\infty \alpha_{kj}\beta_j\in \ell^2$ for any $x=(\beta_j)\in \ell^2$.
Given two matrices $A=(\alpha_{kj})$ and $B=(\beta_{kj})$ with complex entries, their Schur product  is defined by $A*B= (\alpha_{kj}\beta_{kj})$ and endows the space $\B(\ell^2)$  with a structure of Banach algebra, that is  $A*B \in \B(\ell^2)$ whenever  $A, B\in \B(\ell^2)$  (see \cite{Schur}, \cite[Proposition 2.1]{Be} or \cite[Theorem 2.20]{PP}). Moreover
\be \label{ts} \|A*B \|_{\B(\ell^2)}\le \|A \|_{\B(\ell^2)}\|B \|_{\B(\ell^2)}.\ee
Now a  matrix $A=(\alpha_{kj})$ is said to be a Schur multiplier, to be denoted by $A\in {\mathcal M}(\ell^2)$, whenever
$A* B\in \B(\ell^2)$  for any $B\in \B(\ell^2)$ and we write
$$\|A\|_{\mathcal M(\ell^2)}=\sup\{ \|A*B\|_{\B(\ell^2)}: \|B\|_{\B(\ell^2)}\le 1\}.$$
In particular, Schur's result establishes that $\B(\ell^2)\subseteq {\mathcal M}(\ell^2)$.

Operators in $\B(\ell^2)$ and multipliers in ${\mathcal M}(\ell^2)$ are well understood for Toeplitz matrices. Let us denote by $\mathcal T$ the space of matrices with constant diagonals, $A=(\alpha_{kj})$ with  $\alpha_{kj}=\gamma_{j-k}$ for a  given sequence of complex numbers $(\gamma_l)_{l\in \Z}$. The characterization of Toeplitz matrices which define bounded operators in $\B(\ell^2)$ goes back to work of Toeplitz in \cite{T}. It can be seen that $\mathcal T\cap \B(\ell^2)$ can be identified with $L^\infty(\T)$, meaning that
   a Toeplitz matrix $A=(\alpha_{kj})$ belongs to  $\B(\ell^2) $ if and only if there exists $f\in L^\infty(\T) $ such that $\alpha_{kj}=\hat f(j-k)$ for each $ k,j\in \N $. Furthermore
\be \label{tt} \|A\|_{\B(\ell^2)}=\|f\|_{L^\infty(\T)}. \ee
The space $\mathcal C(\ell^2)$ is defined in \cite{PP} as those matrices in $\B(\ell^2)$  such that $\sigma_n(A)$ (see definition below) converges to $A$ in $\B(\ell^2)$. It is shown  (see \cite[Remark 3.2]{PP}) that $\mathcal C(\ell^2)\cap \mathcal T$ can be identified with $C(\T)$.

Similarly $\mathcal T\cap \mathcal M(\ell^2)$ can be identified with the space of regular Borel measures $M(\T)$. It was G. Bennet in \cite{Be} who showed that
  a Toeplitz matrix $A=(\alpha_{kj})$ belongs to $\M(\ell^2) $ if and only if there exists $\mu\in M(\T)$ such that $\alpha_{kj}=\hat\mu(j-k) $ for $k,j\in \N$. Furthermore
  \be \|A\|_{\mathcal M(\ell^2)}= \|\mu\|_{M(\T)}.\label{tb} \ee

 The space  $\mathcal L^1(\ell^2)$ is also defined in \cite[Definition 3.7]{PP}, as those matrices in  $\mathcal M(\ell^2)$ such that $\sigma_n(A)$ converges to $A$ in  $\mathcal M(\ell^2)$. In this case  (see \cite[Remark 3.14]{PP})  $\mathcal L^1(\ell^2)\cap \mathcal T$  can be identified with  $L^1(\T)$.

The reader is also referred to  \cite{AP, Be, BG, PP} for the proofs of the above results.

In this paper we continue the study of certain operator-valued versions of Schur multipliers initiated by the authors (see \cite{BB,BB2}).  Throughout the paper  $(H, \|\cdot\|)$  stands for  a separable Hilbert space and we use the notations $\ell^2(H)$ for the space of sequences ${\bf x}=(x_n)$ with  $x_n\in H$ such that $\|{\bf x}\|_{\ell^2(H)}=(\sum_{n=1}^\infty \|x_n\|^2)^{1/2}<\infty$. In the sequel we write $ \la \cdot,\cdot\ra$ and $ \ll\cdot,\cdot \gg $ for the scalar products in  $H$  and $\ell^2(H)$ respectively, that is $\ll\x,\y\gg= \sum_{n=1}^\infty \la x_n,y_n\ra$
 and we use the notation $x\ej=(0,\cdots,0,x,0,\cdots) $ for the element in $\l2$ in which $x\in H$ is placed in the $j$-th coordinate for $j\in \N$. As usual $c_{00}(H)=span\{x\ej:x\in H, j\in \N\}$.

We denote by  $\B(H)$ the space of bounded linear operators on $H$. Basic examples are the rank one operators given for each $x,y\in H$ by   $x\otimes y(z)=\la z,x\ra y$ for $ z\in H.$
Given a matrix $\A =(\Tkj)$ with entries $\Tkj\in \B(H)$  and $\x\in c_{00}(H)$, we  write $\A \x$ for the sequence $(\sum_{j=1}^\infty\Tkj(x_j))_k$.
We say that $\A\in \B(\ell^2(H))$ if the map $\x \to \A\x$ extends to a bounded linear operator in $\ell^2(H)$, that is there exists $C>0$ such that
$$\left(\mathlarger{\mathlarger{\sum}}_{k=1}^\infty \Norm{\sum_{j=1}^\infty \Tkj(x_j)}^2\right)^{1/2}\le C \Big(\sum_{j=1}^\infty \|x_j\|^2\Big)^{1/2}.$$

We shall write
$$\|\A\|_{\Bl2}=\inf\{C\ge 0: \|\A \x \|_{\l2}\le C \|\x\|_{\l2}\}.$$
Given two matrices $\A =(\Tkj)$ and $\bB=(\Skj)$ with entries $\Tkj, \Skj\in \B(H)$ we define the Schur product $$\A*\bB= (\Tkj\Skj)$$ where $\Tkj\Skj$ stands for  composition of the operators $\Tkj$ and $\Skj$.
Contrary to the scalar-valued case, this product is not commutative.

Given a matrix $\A =(\Tkj)$, we say that $\A$ is a right Schur multiplier (respectively left Schur multiplier), to be denoted by $\A\in \mathcal M_r(\l2)$ (respectively $\A\in \mathcal M_l(\l2)$ ), whenever $\bB* \A\in \Bl2$
 (respectively $\A* \bB\in \Bl2$ ) for any $\bB\in \Bl2$.
We shall write
$$\|\A\|_{\Ml2}=\inf\{C\ge 0: \|\bB * \A  \|_{\Bl2}\le C \|\bB\|_{\Bl2}\}$$
and
$$\|\A\|_{\mathcal M_l(\l2)}=\inf\{C\ge 0: \|\A*\bB   \|_{\Bl2}\le C \|\bB\|_{\Bl2}\}.$$
We say that $\A$ is a Schur multiplier whenever $\A\in \mathcal M_l(\l2)\cap \mathcal M_r(\l2)$ and we set
$$\|\A\|_{\mathcal M(\l2)}=\max \{\|\A\|_{\mathcal M_l(\l2)},\|\A\|_{\mathcal M_r(\l2)}\}.$$

 Denoting by $\A^*$ the adjoint matrix given by  $\Skj=T^*_{jk}$ for all $k,j\in \N$,  one easily sees that $\A\in \Bl2$ (respectively $\A\in \mathcal M_l(\l2)$) if and only if $\A^*\in \Bl2$ with $\|\A\|_{\Bl2}=\|\A^*\|_{\Bl2}$ (respectively  $\A^*\in \mathcal M_r(\l2)$ and $\|\A\|_{\mathcal M_l(\l2)}= \|\A^*\|_{\mathcal M_r(\l2)}$).

It was shown in \cite[Theorem 4.7]{BB} that $\Bl2\subset \mathcal M_l(\l2)\cap \mathcal M_r(\l2)$.
Moreover
\be\label{f1}
\|\A\|_{\mathcal M(\l2)}\le \|\A\|_{\Bl2}.
\ee

We shall use the notation  $\mathcal T(H)$  for the  set of Toeplitz matrices, that is those matrices such that ${T}_{k,j}=T_{j-k}$ for $k,j\in \N$ and $T_l\in \B(H)$ for $l\in \Z$  and  $\mathcal U(H)$ for upper triangular matrices whose entries are operators.
The reader is referred to \cite{BB,BB2} for the analogues of the previous results  on $\Bl2\cap \mathcal T(H)$ and $\Ml2 \cap \mathcal T(H)$ using vector-valued measures. Here we shall consider certain subspaces of $\Bl2$ and $\Ml2$ and we shall avoid the use of vector-valued measures to make the paper self contained.

Throughout the rest of the paper, we write
 $\A=(\Tkj)$ where $\Tkj\in \B(H)$ and we denote by $\Rk$, $\Cj$ and $\Dl$  for $k,j\in \N$ and $l\in \Z$ the matrices consisting of the $k$-row, the $j$-column and $l$-diagonal respectively, that is to say
$$\Rk=(\Tkj)_{j=1}^\infty, \quad \Cj=(\Tkj)_{k=1}^\infty, \quad \Dl=({T}_{k,k+l})_{k=-\min\{l,0\}+1}^\infty.$$

 In \cite{BB2} the class  $ \Cl2$, called ``continuous matrices", with entries in the space $\B(H)$ was introduced and showed to play an important role in the study of Schur multipliers. Here we shall follow a similar approach to define the notion of ``integrable matrices", based upon the notion of ``polynomial" (see \cite[Definition 1.4]{BB2}).
 Given a matrix $\A =(\Tkj)$ with entries $\Tkj\in \B(H)$ we say that $\A$ is a ``polynomial", in short $\A\in\Pl2$, whenever there exist $N, M\in \N$ such that
$\A=\sum_{l=-N}^M\Dl$ and  \be \label{hip0} \sup_{k,j} \|\Tkj\|<\infty.\ee

Notice that  if $\A= (T_{k,j})\in \mathcal M_r(\l2)\cup \mathcal M_l(\l2)$ then
\be\label{f0}\sup_{k,j}\|\Tkj\|\le \min\{\|\A\|_{\mathcal M_r(\l2)}, \|\A\|_{\mathcal M_l(\l2)}\}.\ee

This follows easily using that for $x,y\in H$ and $T\in \B(H)$, one has
$$T (x\otimes y)= x\otimes T(y), \quad (x\otimes y)T= T^*(x)\otimes y.$$
Selecting $\x= x\ej$ and $\y=y\ek$ for some $x,y\in H$, one easily sees that (\ref{f0}) holds.
This shows that condition (\ref{hip0}) is needed for any polynomial to define a multiplier.
\begin{defi}
We define $\Lll2$ (respectively $\Lrl2$) as the closure of $\Pl2$ in $\mathcal M_l(\ell^2(H))$ (respectively $\mathcal M_r(\ell^2(H))$). We use $\Ll2=\Lll2\cap \Lrl2$.
\end{defi}

The paper is divided into two sections. In the first one we analyze the previous definition, presenting several examples in this class and getting an equivalent formulation using Schur product  with Toeplitz matrices given by summability kernels, namely it is shown that $\A\in  \Lrl2$ if and only if $P_r(\A)$ converges to $\A$ in $\Ml2$ or $\sigma_n(\A)$ converges to $\A$ in $\Ml2$ where $P_r(\A)$ and $\sigma_n(\A)$ stand for the Schur product with matrices given by the Poisson or the F\'ejer kernels (see definition below). In Section 3 we study the properties of $\A$ in terms of the properties of certain vector-valued functions related to $\A$. There are two procedures to be considered: the first one consists in  defining a matrix-valued function $f_\A(t)={\bf M}_t* \A$ where ${\bf M}_t=(e^{i(j-k)t})$ for any matrix $\A$ and the second one  in defining a Toeplitz matrix $\A_{\bf f}= (\widehat{\bf f}(j-k))$ for each operator-valued function ${\bf f}$. We show that $t\to f_{\A}(t)$ is continuous as a $\Ml2$-valued function only in the case that $\A\in \Lrl2$, and that $\A_{\bf f}\in \Lrl2$ whenever ${\bf f}\in L^1(\T, \B(H))$. Finally, also the situation of upper triangular matrices and its relationship with Hardy spaces is presented in the last subsection.

\section {Matrices in $\Ll2$}

Let us start computing the norm of $\Dl$, $\Rk$ and $\Cj$ in the space of Schur multipliers.

\begin{ejem} \label{ej2} Let  $\A =(\Tkj)$ and let $l\in \Z$ and $k,j\in \N$. Then

(i) $\Dl\in \mathcal M(\l2)$ iff $\sup_k \|T_{k,k+l}\|<\infty$ iff $\Dl\in \Bl2$. Moreover
$$\|\Dl\|_{\mathcal M(\l2)}=\|\Dl\|_{\Bl2}=\sup_{k\ge -\min\{l,0\}+1} \|T_{k,k+l}\|.$$

(ii) $\Cj\in \mathcal M_l(\l2)$ iff $\sup_k \|T_{k,j}\|<\infty.$ Moreover
$$\|\Cj\|_{\mathcal M_l(\l2)}=\sup_k \|T_{k,j}\|.$$

(iii) $\Rk\in \mathcal M_r(\l2)$ iff $\sup_j \|T_{k,j}\|<\infty.$ Moreover
$$\|\Rk\|_{\mathcal M_r(\l2)}=\sup_j \|T_{k,j}\|.$$

\end{ejem}
\begin{proof}
 (i) It is straightforward to see that $\|\Dl\|_{\Bl2}=\displaystyle\sup_{k\ge -\min\{l,0\}+1} \|T_{k,k+l}\|$.
 Notice that for $\bB=(S_{kj})$ one has that $ \Dl *\bB= \Dl'$ where $\Dl'= T_{k,k+l}S_{k,k+l}$. Hence
 $$\|\Dl *\bB\|_{\Bl2}= \sup_{k\ge -\min\{l,0\}+1} \|T_{k,k+l}S_{k,k+l}\|\le \sup_{k\ge -\min\{l,0\}+1} \|T_{k,k+l}\| \|\bB\|_{\Bl2}.$$
 Similarly for $\bB*\Dl$ and (i) holds.

 (ii) Note that $\|\Cj\|_{\Bl2}=\sup_{\|x\|=1} (\sum_{k=1}^\infty\|T_{kj}(x)\|^2)^{1/2}<\infty.$
Notice that for $\bB=(S_{kj})$ one has that $ \Cj *\bB= \Cj'$ where $\Cj'= T_{k,j}S_{k,j}$. Hence
 \ba\|\Cj *\bB\|_{\Bl2}&=& \sup_{\|x\|=1} (\sum_{k=1}^\infty\|T_{k,j}S_{k,j}(x)\|^2)^{1/2}\\
 &\le& \sup_k \|T_{k,j}\|\sup_{\|x\|=1} (\sum_{k=1}^\infty\|S_{k,j}(x)\|^2)^{1/2}\\
 &\le&\sup_k \|T_{k,j}\| \|\bB\|_{\Bl2}.\ea

 (iii) follows from (ii) by taking adjoints.
\end{proof}

\begin{ejem} Since $\Bl2 \subset \mathcal M(\l2)$ we clearly have $\Cl2\subset \Ll2$. In particular,  if $\A=\sum_l \Dl$ such that $\sum_l \|\Dl\|_{\Bl2}<\infty$ then  $\A\in\Ll2$.\end{ejem}

\begin{ejem}
Let $\x=(x_j)$ and $\y=(y_k)$ belong to $\l2$.
Then  $$(\x\otimes \y)(\z)=\ll \z,\x\gg \y, \quad \z\in \l2$$
corresponds to the matrix $\A =(x_j\otimes y_k) $ and belongs to $\Ll2$.
Moreover  $\|\x\otimes \y\|_{\Ll2}\le \|\x\|_{\l2}\|\y\|_{\l2}.$
\end{ejem}
\begin{proof} It is clear that $\x\otimes \y\in \Bl2$ and $\|\x\otimes \y\|_{\Bl2}=\|\x\|_{\l2}\|\y\|_{\l2}.$
Let $\x\otimes \y= (\Tkj)$. Note that $$\la T_{k,j}x,y\ra= \ll \x\otimes \y(x \ej),y\ek\gg= \la x,x_j\ra \la y,y_k\ra.$$ This gives that $T_{k,j}=x_j\otimes y_k$ for $k,j\in \N$.
Now taking into account that $\x_N=(x_1,\cdots,x_N,0,\cdots)$ converges to $\x$ in $\l2$ we obtain that $\x\otimes \y$ is the limit  in $\mathcal M(\l2)$ of $\x_N\otimes \y_N$ as $N\to \infty$ . Since $\x_N\otimes \y_N\in \Pl2$, one has the result.
\end{proof}

Given $\eta\in M(\T)$ we shall denote by ${\bf M}_\eta$ the Toeplitz matrix given by  $${\bf M}_\eta= \left(\hat\eta(j-k) Id\right)_{k,j}\in \mathcal T(H)$$
where $Id:H\to H$ is the identity operator.
  The cases $\eta=\delta_{-t}$  or $d\eta= f dt$ with $f\in L^1(\T)$ will be denoted by ${\bf M}_t$ and ${\bf M}_f$ respectively, that is ${\bf M}_t=(e^{i(j-k)t} Id)$ and  ${\bf M}_f=(\hat f(j-k) Id)$.

  An easy procedure to generate matrices in $\mathcal M(\l2)$ is the following one (see \cite[Proposition 3.2]{BB2}) showing that
if  $A= (a_{kj})\in \M(\ell^2)$ and $T\in \B(H)$, then
$\A=( a_{k,j}T)\in {\mathcal M}(\ell^2(H))$ and
\be \label{f1} \|\A\|_{\mathcal M(\ell^2(H))}= \|A\|_{\M(\ell^2)} \|T\|_{\B(H)}.\ee
In particular one can produce the following examples.
\begin{ejem}

(i) If $A= (a_{kj})\in \mathcal L^1(\ell^2)$ and $T\in \B(H)$ then
$\A=( a_{k,j}T)\in {\mathcal L}^1(\ell^2(H))$.

(ii) If $f\in L^1(\T)$ and  $T\in \B(H)$ then
$\A=( \hat f(j-k)T)\in {\mathcal L}^1(\ell^2(H))$.
\end{ejem}

 Recall that a family $\lbrace{k_\e\rbrace}_{\e>0}\subset L^1(\mathbb{T})$ is called a ``summability kernel'' if it satisfies

1) $\frac{1}{2\pi}\int_{-\pi}^\pi k_\e(t) dt=1$ for all $\e>0$.
\medskip

2) $\sup_{\e>0}\frac{1}{2\pi}\int_{-\pi}^\pi |k_\e(t)| dt=C<\infty$.
\medskip

3) $\forall 0<\delta<\pi$ one has $\frac{1}{2\pi}\int_{\delta\leq|t|\leq\pi} k_\e(t) dt\xrightarrow[\e\to 0]{}0$.
\bigskip

Classical examples to be used in the sequel are
the F\'ejer kernel (for $\e=\frac{1}{n}$) $$K_n(t)=\mathlarger{\mathlarger{\sum}}_{k=-n}^n\left( 1- \frac{|k|}{n+1}\right)e^{ikt}$$ and the Poisson kernel (for $\e=1-r$) $$ P_r(t)=\sum_{k\in \Z} r^{|k|}e^{ikt}.$$ We shall use the notation
$\sigma_n(\A)= {\bf M}_{K_n}* \A$ and
$ P_r(\A)=  {\bf M}_{P_r}* \A$ for $\A=(T_{k,j})$.

Observe that under the assumption (\ref{hip0}) one has $\sigma_n(\A)\in {\mathcal P}(\l2)$ and $P_r(\A)\in \mathcal{C}(\l2)$, since $\sup_l \|{\bf D}_l\|_{\Bl2}<\infty$.

It was shown in \cite[Proposition 3.4]{BB2} that given
 a matrix $\A$  with entries in $\mathcal{B}(H)$ and  a summability kernel $\lbrace{k_n\rbrace}$, if we denote $M_n(\A)= {\bf M}_{k_n}* \A$ then
 \be \label{equi00} \A\in \Bl2\;\Leftrightarrow\; \sup_n\normbh{M_n(\A)}<\infty\ee
\be \label{equi0} \A\in \Ml2\;\Leftrightarrow\; \sup_n\normmh{M_n(\A)}<\infty\ee
and similar result for left Schur multipliers.

We shall see now that the space of those matrices $\A\in \Ml2$ such that $M_n(\A)$ converges to  $\A$ in $\Ml2$ corresponds to $\Lrl2$. Next proof follows the same arguments as \cite[Theorem 4.4]{BB2} but we include it for the sake of completeness.
\begin{teor} \label{thm:caract_L1}
Let $\A$ be a matrix whose entries are in $\mathcal{B}(H)$. The following are equivalent:

1) $\A\in \Lrl2$.

2) $\lim_{n\to \infty} M_n(\A)=\A$ in $\Ml2$
where $M_n(\A)=\bold{M}_{k_n}\ast \A$ and $\lbrace{k_n\rbrace}\subseteq L^1(\T)$ is a summability kernel.

3) $\lim_{n\to \infty} \sigma_n(\A)=\A$ in $\Ml2$.

4) $\lim_{r\to 1} P_r(\A)=\A$ in $\Ml2$.

\end{teor}

\begin{proof}

1)$\Rightarrow$ 2). Let $\varepsilon>0$, and select $\bold{P}=(S_{k,j})_{k,j}=\sum_{l=-N}^{N}\Dl\in\Pl2$ such that $\normmh{\A-\bold{P}}<\varepsilon/3C$ where
$C=\sup_n \|k_n\|_{L^1(\T)}\ge 1$ Then, using part (i) in Example \ref{ej2} ,
\ba
\normmh{M_n(\bold{P})-\bold{P}}&=&\Norm{\sum_{l=-N}^{N}(\hat{k_n}(l)-1)\Dl}_{\Ml2}\\
&\leq&\sup_{k,j}\norm{S_{k,j}}\cdot (2N+1)\cdot \max_{|l|\leq N}|\hat{k_n}(l)-1|\ea
 Since $\lbrace{k_n\rbrace}$ is a summability kernel, one has that $\hat k_n(l)\to 1$ as $n\to\infty$ $\forall l\in \Z$. So, we can choose $n_0\in\mathbb{N}$ such that $|\hat{k}_n(l)-1|<\frac{\varepsilon}{3(2N+1) \sup_{k,j}\norm{S_{k,j}}}$ $\forall n\geq n_0$ and $\forall |l|\leq N$. Hence, $\norm{M_n(\bold{P})-\bold{P}}_{\Ml2}<\varepsilon/3$.   Finally, for $n\ge n_0$,
 \begin{eqnarray*}
\normmh{M_n( \A)-\A}
&\leq&\normmh{\bold{M}_n\ast (\A-{\bf P})}\\
&+&\normmh{\bold{M}_n\ast \bold{P}-\bold{P}}+\normmh{\bold{P}-\A} \\
&\leq&\normmh{\bold{M}_n}\cdot\normmh{\A-\bold{P}}+\varepsilon/3+\varepsilon/3 \\
&\leq&\norm{k_n}_{L^1}\cdot \varepsilon/3C+\varepsilon/3+\varepsilon/3=\varepsilon. \\
\end{eqnarray*}

The implications 2)$\Rightarrow$ 3) \& 4) and 3)$\Rightarrow$ 1) are obvious since the F\'ejer  and Poisson kernels are summability kernels
and $\sigma_n(\A) \in \Pl2$.

 4) $\Rightarrow$ 1). Note that $P_r(\A)\in \Lrl2$ for each $0<r<1$ since the series $\sum_{l\in \Z} \Dl r^{|l|}$ is absolutely convergent in $\Lrl2$. Hence its limit also belongs to $\Lrl2$.
\end{proof}

\begin{coro}
\label{ej3} Let  $\A =(\Tkj)$ satisfying (\ref{hip0}) and  $j\in \N$. Then
 $\Cj\in \Lll2$ iff $\lim_{k\to \infty} \|T_{k,j}\|=0.$
\end{coro}
\begin{proof} Notice that one has
 $$(\Cj-\sigma_n(\Cj))_{k,j}= \frac{|k-j|}{n+1} T_{k,j} \quad k\le j+n$$
$$(\Cj-\sigma_n(\Cj))_{k,j}=  T_{k,j} \quad k > j+n.$$
The result now  follows from Example \ref{ej2} and Theorem \ref{thm:caract_L1}.
\end{proof}

\bigskip

Of course, if $\A=\sum_l \Dl$ satisfying (\ref{hip0}) one has that $\Dl\in \Ll2$ for each $l\in \Z$ since $\Dl-\sigma_n(\Dl)=\frac{|l|}{n+1}\Dl$ for $n\ge |l|$.
Moreover, if $\A\in \Ml2$ then for each $l\in \Z$
\be \label{inenor} \|\Dl\|_{\Bl2}=\|\Dl\|_{\Ml2}\le \|\A\|_{\Ml2}.\ee

\begin{prop}(Riemann-Lebesgue lemma)
 If $A=\sum_l \Dl\in \Lrl2$, then
 $$\|\Dl\|_{\Bl2}\xrightarrow[|l|\to\infty]{}0.$$
\end{prop}
\begin{proof} For each $\e>0$ select $n_0\in \N$ such that $\norm{\sigma_{n_0}(\A)-\A}_{\Ml2}<\e$. For $|l|>n_0$ we have that $\sigma_{n_0}(\Dl)=0$. Hence, using (\ref{inenor}), we conclude
\ba
\norm{\Dl}_{\Bl2}& \leq& \norm{\sigma_{n_0}(\Dl)-\Dl}_{\Ml2}\\
&\le&\norm{\sigma_{n_0}(\A)-\A}_{\Ml2}<\e.
\ea
This gives the result.
\end{proof}

Recall that  $\mathcal A(\l2)$ is the analogue to the Wiener algebra, that is matrices $\A=\sum_{l\in \Z} \Dl$ such that $\sum_{l\in \Z} \|\Dl\|_{\Bl2}<\infty$. Since $P_r(\A)\in \mathcal A(\l2)$ for any $\A\in \Ml2$ we obtain the following corollary.\begin{coro} $\mathcal A(\l2)$ and $\Cl2$ are dense in $\Ll2$.
\end{coro}

\begin{nota} $\Lrl2$ is a right ideal of $\Ml2$, that is to say if $\A\in \Lrl2$ and ${\bf B}\in \Ml2$ then
$ {\bf B}\ast {\bf A}\in \Lrl2.$
\end{nota}

\begin{defi} We write $(\Bl2,\Cl2)_l$  for the set of matrices $\A$ such that
$$ \A\ast {\bf B}\in \Cl2 \quad \forall {\bf B}\in \Bl2.$$
Similar definitions can be given for $(\Cl2,\Cl2)_l$ and for right Schur multipliers.
\end{defi}

In \cite[Theorem 4.6]{BB2} it was shown that
$ \A\in \M_l(\l2)$ (respectively $\A\in \M_r(\l2)$) if and only if $\A\in (\Cl2,\Cl2)_l$
(respectively $\A\in (\Cl2,\Cl2)_r$).

\begin{coro} $\Lrl2\subset  (\Bl2,\Cl2)_r$ and similar result for left multipliers.
\end{coro}
\begin{proof}
Let us assume that $\A\in\Lrl2$ and  $\bold{B}\in \Bl2$. Since $\sigma_n( \bold{B}\ast \A)={\bf B}\ast\sigma_n(\A)$
we have
$$ \|\sigma_n( \bold{B}\ast \A)-  \bold{B}\ast \A\|_{\Bl2}\le \|(\sigma_n(\A)- \A)\|_{\Ml2}\| \bold{B}\|_{\Bl2}$$
and the result follows taking limits as $n\to \infty$.
\end{proof}

\section{Matrices versus functions}
Given a complex Banach space $X$ we  denote  by $P(\T,X)$, $C(\T,X)$ and $L^p(\T, X)$   the spaces of $X$-valued trigonometric polynomials, $X$-valued continuous functions and  $X$-valued strongly measurable functions with $\|f\|_{L^p(\T,X)}=(\int_0^{2\pi} \|f(e^{it})\|_X^p\dt)^{1/p}<\infty$ for $1\le p \le \infty$ (with the usual modification for $p=\infty$). We use the notations $\mathcal M(\T,X)$ and $M(\T,X)$ for regular $X$-valued measures and those with bounded variation respectively.  We refer to \cite{B, DFS, DU, HNVW}  for the results  on vector-valued Fourier analysis, vector measures  and projective tensor products to be used in the sequel.

To each regular vector measure $\mu\in \mathcal M(\T, \B(H))$ defined on the Borel sets of $\T$ and with values in $\B(H)$ we can associate a Toeplitz matrix
$\A_\mu=(T_{k,j})$ given by
\be
T_{k,j}=\hat\mu(j-k), \quad k,j\in \N
\ee
where $\hat\mu(l)= \int_0^{2\pi} e^{-il t}d\mu(t)\in \B(H)$ for $l\in \Z$.

If $d\mu ={\bf g}dm$ for a given function ${\bf g}:\T \to \B(H)$ we simply denote it by $\A_{\bf g}$. In particular if ${\bf g}\in P(\T,\B(H))$ then $\A_{\bf g}\in \Pl2$.

 The following  operator-valued function was introduced in \cite{BB2} for  each matrix  $\A=(\Tkj)$: $$f_\A(t)=(e^{i(j-k)t}\Tkj), \quad t\in [-\pi,\pi).$$
Clearly if  $\A\in \Pl2$ one has
$$f_\A(t)=\sum_{l\in \Z} \Dl e^{ilt}\in  P(\T, \Bl2).$$
Taking into account that $f_\A(t)={\bf M}_t*\A$ and that ${\bf M}_t\in \mathcal M(\l2)\cap \mathcal T(H)$ with $\|{\bf M}_t\|_{\mathcal M(\l2)}=1$, then  $ f_\A(t)$ takes values into $\Bl2$, $\Ml2$ or $\Lrl2$ whenever $\A$ belongs to $\Bl2$, $\Ml2$ or $\Lrl2$ respectively and with the same norm for each $t\in [0,2\pi)$.
In particular, for $\A\in \Ml2$ one has
\be \label{equi}
\sup_{t\in [0,2\pi)} \|f_\A(t)\|_{\Ml2}= \|\A\|_{\Ml2}.
\ee
It  was shown that in general if $\A\in \Bl2$, then $t\to f_{\A}(t)$ was not strongly measurable as a $\Bl2$-valued mapping  (see \cite[Proposition 4.2]{BB2}). The fact that the function $t\to f_{\A}(t)$ is continuous as a $\Bl2$-valued function is actually equivalent to $\A\in \Cl2$ (see \cite[Proposition 4.4]{BB2}).
We shall analyze now the properties of such a function as a multiplier-valued function.

\begin{prop} (i) There exists $\A\in \mathcal M(\l2)$ such that $f_\A$ is not strongly measurable as a $\mathcal M(\l2)$-valued function.

(ii) There exist $\A\in \mathcal M_l(\l2)$  and ${\bf B}\in \Bl2$ such that the map $t\to f_\A(t)*{\bf B}$ is not strongly measurable as a $\Bl2$-valued function.
\end{prop}

\begin{proof}
(i) It suffices to select $\A={\bf 1}$ where we use $\bf 1$ for the unit element in $\mathcal M(\l2)$ given by ${\bf 1}_{k,j}= Id$ for the identity operator $Id:H\to H$.
Clearly $f_{\bf 1}(t)={\bf M}_t=(e^{i(j-k)t}Id)$ is not $\mathcal M(\l2)$-valued strongly measurable. Indeed, from (\ref{f1}), one has that
$$\|f_{\bf 1}(t)- f_{\bf 1}(s)\|_{\mathcal M(\l2)}=\|\delta_{-t}-\delta_{-s}\|_{M(\T)}=2, \quad t\ne s$$
and therefore the range of $f_{\bf 1}$ is not separable.

(ii) Select $\A={\bf 1}$ and ${\bf B}=(\Tkj)$ where $\Tkj=0$ for each $j\ne 2k$ and $T_{k,2k}=Id$ for $k\in \N$, and note the fact that $f_{\bf 1}* {\bf B}= f_{\bf B}$ which according to \cite[Proposition 4.2]{BB2}  is not strongly measurable with values in $\Bl2$.
\end{proof}

\begin{prop}  Let $\A=(\Tkj)\in \M_l(\l2)$ and ${\bf B}=(\Skj)\in \Bl2$ .

If either $\A\in (\Bl2, \Cl2)_l$ or ${\bf B}\in \Cl2$ then $t\to f_{\bf A}(t)* {\bf B}$  is continuous with values in $\Bl2$.

In particular   if $\x,\y\in \l2$ then the map
$$f_\A(t)\ast(\x\otimes \y)=\Big(e^{i(j-k)t} x_j\otimes \Tkj(y_k)\Big)_{k,j}$$
is continuous from $\T$ into $\Bl2$.
\end{prop}
\begin{proof}
 Both cases follow invoking  \cite[Proposition 4.4]{BB2} since $f_{\bf A}(t)* {\bf B}= f_{\A*{\bf B}}(t)$ and $\A*{\bf B}\in \Cl2$ in each situation.
\end{proof}

Let us now give another characterization of matrices in $\Lll2$.

\begin{teor} \label{thm:caract_cont}
Let $\A$ be a matrix whose entries are in $\mathcal{B}(H)$. Then
 $\A\in \Lll2$ iff
 $t\to f_\A(t)$ is a $\mathcal M_l(\l2)$-valued continuous function.
\end{teor}

\begin{proof}
Note that $\|\A\|_{\mathcal M_l(\l2)}= \|f_{\A}(t)\|_{\mathcal M_l(\l2)}$ for any $t\in \T$. On the other hand, $\sigma_n(f_\A(t))=f_{\sigma_n(\A)}(t) \in P(\T, \Bl2)$ and
\ba
\sup_t\|f_{\sigma_n(\A)}(t)-f_\A(t)\|_{\mathcal M_l(\l2)}&=&\sup_t\|f_{\sigma_n(\A)-\A}(t)\|_{\mathcal M_l(\l2)}\\
&=&\|\sigma_n(\A)-\A\|\|_{\mathcal M_l(\l2)}.\ea
It is well known that $f_\A\in C( \T, {\mathcal M_l(\l2)})$ iff $\sigma_n(f_{\A})$ converges to  $f_\A$ in $C( \T, {\mathcal M_l(\l2)})$, which shows the equivalence between both conditions.
\end{proof}

\subsection{The Toeplitz case}

It was observed  previously that  $\A_f=(\hat f(j-k)T)\in \Ll2\cap \mathcal T(H)$ whenever $f\in L^1(\mathbb{T})$ and $T\in \mathcal{B}(H)$. We shall first show that this is actually true for operator-valued integrable functions. This result might be obtained from the
inclusion $M(\T, \Bl2)\subset \Ml2 \cap\mathcal T(H)$, but we shall give a direct proof in this case.

\begin{prop}
\label{lema:norma}
Let ${\bf f}\in L^1(\mathbb{T},\mathcal{B}(H))$, and consider $\A_{\bf f}=(\widehat{{\bf f}}(j-k))=(T_{j-k})$. Then $\A_{\bf f}\in \Ll2$ with
$$\|\A_{\bf f}\|_{\mathcal M(\l2)}\le \norm{{\bf f}}_{L^1(\T,\mathcal{B}(H))}$$
\end{prop}

\begin{proof}
  Recall that $\B(H)=(H\hat\otimes H)^*$ by means of the formula
  $T(x\otimes y)=\la Tx,y\ra$.
Assume first that ${\bf f}\in  P(\T, \mathcal{B}(H))$ and let $\x=(x_j),\y=(y_k)\in \ell^2(H)$  and ${\bf B}=(S_{k,j})\in \Bl2$. We can write
\begin{align*}
|\ll (\A_{\bf f}*&{\bf B}) \x,\y\gg|=\left|\sum_{k,j}\presc{T_{k,j}S_{k,j} x_j}{y_k}\right|=\left|\sum_{k,l}\presc{T_l S_{k,k+l}x_{l+k}}{y_k}\right|\\
&=\left|\sum_l\ T_l\Big(\sum_k S_{k,k+l}x_{l+k}\otimes y_k\Big)\right|\\
&=\left|\bigintss_0^{2\pi}\left(\sum_l\hat {\bf f}(l) e^{ilt}\right)\left(\mathlarger{\mathlarger{\sum}}_l \left(\sum_k S_{k,k+l} x_{l+k}\otimes y_k\right) e^{-ilt}\right)\frac{dt}{2\pi}\right|\\
&=\left|\bigintss_0^{2\pi}{\bf f}(t)\left( \sum_k\left(\sum_j  S_{k,j}x_{j}e^{-ijt}\right)\otimes y_k e^{ikt}\right)\frac{dt}{2\pi}\right|\\
&\leq\bigintss_0^{2\pi} \norm{{\bf f}(t)}_{\mathcal{B}(H)}\Norm{\sum_k\left(\sum_j  S_{k,j}x_{j}e^{-ijt}\right)\otimes y_k e^{ikt}}_{H\hat\otimes H}\frac{dt}{2\pi}\\
&\leq\bigintss_0^{2\pi} \norm{{\bf f}(t)}_{\mathcal{B}(H)}\sum_k\|\sum_j  S_{k,j}x_{j}e^{-ijt}\|_H\|y_k \|_H\frac{dt}{2\pi}\\
&= \norm{{\bf f}}_{L^1(\T,\mathcal{B}(H))} \sup_{t\in [0,2\pi)} \left(\sum_k\|\sum_j  S_{k,j}x_{j}e^{-ijt}\|^2_H\right)^{1/2}\norm{\y}_{\ell^2(H)}\\
&= \norm{{\bf f}}_{L^1(\T,\mathcal{B}(H))}\|{\bf B}\|_{\Bl2}\norm{\x}_{\ell^2(H)}\norm{\y}_{\ell^2(H)}.
\end{align*}

Therefore $\|\A_{\bf f}\|_{{\mathcal M}_l(\l2)}\leq\norm{{\bf f}}_{L^1(\T,\mathcal{B}(H))}$.
To get $\|\A_{\bf f}\|_{{\mathcal M}_r(\l2)}\leq\norm{{\bf f}}_{L^1(\T,\mathcal{B}(H))}$, just notice that
for ${\bf f}^*(t)=({\bf f}(t))^*$ one has ${\bf f}^*\in L^1(\T, \mathcal B(H))$ with $\norm{{\bf f^*}}_{L^1(\T,\mathcal{B}(H))}=\norm{{\bf f}}_{L^1(\T,\mathcal{B}(H))}$  that $\widehat{{\bf f}^*}(l)= (\hat {\bf f}(-l))^*$ for all $l\in \Z$. Since $\A_{\bf f}^*= \A_{\bf f^*}$ we get the other estimate.

To obtain the general case ${\bf f}\in L^1(\T, \Bl2)$ we can use an approximation argument since polynomials are dense in $L^1(\T, \Bl2)$.
\end{proof}

Recall that for $\Cl2$ we have the following fact.
\begin{lema}
\label{cinfty} (see \cite[Lemma 4.7]{BB2})
If ${\bf f}\in C(\mathbb{T},\mathcal{B}(H))$ then $\A_{\bf f}\in \Cl2$. Moreover
$$\normbh{\A_{\bf f}}= \norm{{\bf f}}_{C(\T,\mathcal{B}(H))}.$$

\end{lema}

\begin{defi}  Given $P\in P(\T, \mathcal B(H))$, say $P(t)= \sum_{l} T_l e^{ilt}$ for some $(T_l)_{l\in \Z}\in c_{00}(\mathcal B(H))$, we denote
$$\|P\|_{L^1_{SOT}}= \sup_{\|x\|=1} \bigintsss_0^{2\pi} \Norm{\sum_{l} T_l(x) e^{ilt}}\frac{dt}{2\pi}.$$
\end{defi}

Next result can be achieved from the inclusion $\Ml2\cap \mathcal T(H)\subset M_{SOT}(\T,\B(H))$, but we give an independent proof based upon Lemma \ref{cinfty}.

\begin{prop}
\label{prp}
Let $(T_l)_{l\in \Z}\in c_{00}(\B(H))$ and $P(t)= \sum_{l} T_l e^{ilt}$.  Then,
$\A_P\in \mathcal{P}(\ell^2(H))$ and
$\norm{P}_{L^1_{SOT}}\le \norm{\A_P}_{\Ml2}.$
\end{prop}

\begin{proof}
 For each ${\bf B}\in \Pl2 \cap\mathcal T(H)$  given by $Q(t)=\sum_l S_l e^{ilt}\in P (\T,\B(H))$, applying Lemma \ref{cinfty} to ${\bf B}*\A_P$ and ${\bf B}$, we can write
\be \label{hipotesis}\sup_{t\in [-\pi,\pi]}\left\|\sum_l S_l T_l e^{ilt}\right\|_{\B(H)}\le \|\A_P\|_{\mathcal{M}_r(\ell^2(H))} \sup_{t\in [-\pi,\pi]}\left\|\sum_l S_l e^{ilt}\right\|_{\B(H)}.\ee
Recall that $L^1(\T, H)\subset (C(\T, H))^*$, that is for $g=\sum_l \hat g(l) \varphi_l\in L^1(\T, H)$ then $\Phi_g(\sum_l x_l\varphi_l)=\sum_l \la x_l, \hat g(l)\ra$ defines a functional in $(C(\T, H))^*$ with $\|g\|_{L^1(\T, H)}= \|\Phi_g\|_{(C(\T, H))^*}$.
Therefore
$$\norm{P}_{L^1_{SOT}}= \sup\{ |\sum_l \la x_l, T_l(x)\ra|: \|x\|=1, \|\sum_l x_l\varphi_l\|_{C(\T, H)}=1\}.$$
Let $x\in H$ with $\|x\|=1$ and $(x_l)\subset H$ such that $\|\sum_l x_l\varphi_l\|_{C(\T, H)}=1$. Define $S_l=x_l\otimes x$ for $l\in \Z$. One has that $\la x_l, T_l(x)\ra x=S_l(T_l x)$ and $\sum_l S_l e^{ilt}\in C(\T,\B(H))$ with
$$\sup_{t\in [-\pi,\pi]}\left\|\sum_l S_l e^{ilt}\right\|_{\B(H)}= \sup_{t \in [-\pi,\pi], \|z\|=1} \left\|\sum_l \la x_l e^{ilt}, z\ra x_0\right\|_{H}=\sup_{t\in [-\pi,\pi]}\left\|\sum_l x_l e^{ilt}\right\|_{H}.$$
From (\ref{hipotesis})   we obtain
\ba
\left|\sum_{l} \la x_l, T_l(x)\ra\right|&\le & \sum_{l}  \|S_lT_l(x)\|_{H}\\
&\le& \left\|\sum_l S_lT_l\right\|_{\B(H)}\\
&\le&\sup_{t\in [-\pi,\pi]}\left\|\sum_l S_lT_le^{ilt}\right\|_{\B(H)}\\
&\le&\|\A_P\|_{{\mathcal M}_r(\l2)} \sup_{t\in [-\pi,\pi]}\left\|\sum_l S_l e^{ilt}\right\|_{\B(H)}\\
&=&\|\A_P\|_{{\mathcal M}_r(\l2)} \sup_{t\in [-\pi,\pi]}\left\|\sum_l x_l e^{ilt}\right\|.
\ea
The result is now complete.
\end{proof}

We write $\tilde L^1_{SOT}(\T,\mathcal B(H))$ for the closure of polynomials under the $\|\cdot\|_{L^1_{SOT}}$.
Combining Proposition \ref{lema:norma} and Proposition \ref{prp} we obtain the following result.
\begin{coro} $L^1(\T, \B(H))\subset \Lrl2 \subset \tilde L^1_{SOT}(\T, \B(H))$.
\end{coro}

\subsection{The upper triangular case}

As usual, if $X$ is a complex Banach space we write $\mathcal H(\D,X)$ for the space of $X$-valued holomorphic functions, $H^\infty(\D, X)$ for the Banach space of bounded analytic functions on the unit disc with values in $X$ and $A(\D,X)$ stands for the disc algebra that is the closure of analytic polynomials in $H^\infty(\D, X)$, with the norm
$$\norm{F}_{H^\infty(\D,X)}=\sup\lbrace{\|F(z)\|\;\;|\; z\in \D\rbrace}.$$
Also we denote by $H^1(\D, X)$ the space of functions $F\in \mathcal H(\D,X)$ such that $$\|F\|_{H^1(\D, X)}=\sup_{0<r<1} \int_0^{2\pi} \|F(re^{it})\|\frac{dt}{2\pi}<\infty.$$
We write $H^1(\T, X)$ for the closure of analytic polynomials under this norm, which turns out to coincide with functions in $f\in L^1(\T, X)$ such that $\hat f(l)=0$ for $l<0.$
It is well known that given $F\in H^1(\D, X)$ and $F_r(z)=F(rz)$ then $F_r\in H^1(\T,X)$. Moreover one has that $F_r\to F$ in $H^1(\D, X)$ if and only if $F\in H^1(\T, X)$. In general $H^1(\T,X)$ does not coincide with $H^1(\D,X)$. The property for that to hold is the so called Analytic Radon-Nikodym property, in short $ARNP$, introduced in \cite{BD}. It is easy to see that $c_0\subset \B(H)$ and then $\B(H)$ fails to have the $ARNP$. In particular $H^1(\T, \B(H))\subsetneq H^1(\D, \B(H))$. Since we only need the basic theory, which extends to the vector-valued setting from the scalar-valued one, we refer to the books  \cite{D,G} for possible results to be used.

\begin{defi} Let $\A=(\Tkj)\in \mathcal U(H)$. Define
$$F_\A(z)=  \sum_{l=0}^\infty \Dl z^l=(z^{(j-k)}\Tkj), \quad |z|<1,$$
\end{defi}

It follows from the definitions that
\be \label{anal} F_\A(re^{it})={\bf M}_{P_r}* {\bf M}_t* \A={\bf M}_{P_r}\ast f_\A(t)=\sum_{l=0}^{\infty}{\bf D}_l r^le^{ilt}.\ee

\begin{nota} If $\A=(T_{k,j})_{k,j}\in \mathcal U(H)$ satisfies the condition (\ref{hip0})
we can guarantee that
 $F_\A(z)=\sum_{l=0}^\infty \Dl z^l$
 is a well defined holomorphic function in ${\mathcal H}(\D, \Bl2).$
\end{nota}

\begin{prop}
Let $\A=(\Tkj)\in \mathcal U(H)$ satisfying (\ref{hip0}).

(i)
 $\A\in \Ml2$ if and only if $F_\A\in H^\infty(\D, \Ml2)$. Moreover $$\|\A\|_{\Ml2}= \|F_\A\|_{H^\infty(\D, \Ml2)}.$$

 (ii) $\A\in (\Bl2, \Cl2)_r$ if and only if ${\bf B}*F_\A\in A(\D, \Bl2)$ for all ${\bf B}\in \Bl2$. Moreover $$\|\A\|_{(\Bl2, \Cl2)_r}= \sup\{\|{\bf B}*F_\A\|_{H^\infty(\D, \Bl2)}:\|{\bf B}\|_{\Bl2}=1\}.$$

 (iii)
 $\A\in \Lrl2$ if and only if $F_\A\in A(\D, \Ml2)$. Moreover $$\|\A\|_{\Ml2}= \|F_\A\|_{A(\D, \Ml2)}.$$

\end{prop}

\begin{proof} (i) From (\ref{anal}) one gets $\|F_\A\|_{H^\infty(\D, \Ml2)}\le \|\A\|_{\Ml2}$. Conversely, use (\ref{equi0}) for $k_n= P_{r_n}$ for a sequence $r_n$ converging to $1$ to obtain $\|F_\A\|_{H^\infty(\D, \Ml2)}=\sup_n\|P_{r_n}(\A)\|_{\Ml2}= \|\A\|_{\Ml2}$.

(ii) It follows from \cite[Theorem 4.12]{BB2} that $F_{{\bf B}*\A}\in A(\D, \Bl2)$ for all ${\bf B}\in \Bl2$ due to the condition ${\bf B}*\A\in \Cl2 \cap {\mathcal U}(H)$.

(iii)
 Using Theorem \ref{thm:caract_cont} we know that $\A\in \Lrl2)$ if and only if $f_\A\in C(\T, \Ml2)$.
Since $F_\A(re^{it})=P_r(f_\A(t))$, invoking part (i) we have that $$\normmh{P_r(\A)-\A}=\|P_r*f_\A- f_\A\|_{C(\T,\Ml2)}$$
which gives the result.
\end{proof}

\begin{defi}
Let $\A=(T_{j-k})\in \mathcal U(H)\cap \mathcal T(H)$, we write $$G_\A(z)=  \sum_{l=0}^\infty T_l z^l, \quad |z|<1.$$
\end{defi}

The assumption  $\sup_{l\ge 0} \|T_l\|<\infty$ gives that $G_\A(z)=\sum_{l=0}^\infty T_l z^l\in {\mathcal H}(\D,\B(H))$. In particular,  for each $0<r<1$
$$(G_\A)_r(e^{it})=G_\A(re^{it})=\sum_{l=0}^\infty T_l r^le^{ilt}\in C(\T, \B(H)).$$

\begin{teor} Let $\A= (T_{j-k})\in {\mathcal U(H)}\cap {\mathcal T(H)} $ with $\sup_{l\ge 0} \|T_l\|<\infty$.

(i) $\A\in \Bl2$ if and only if $G_\A\in H^\infty(\D,\B(H))$.

(ii) $\A\in \Cl2$ if and only if $G_\A\in A(\D,\B(H))$.

(iii) If $G_\A\in H^1(\D,\B(H))$ then $\A\in \Ml2$.

(iv) If $G_\A\in H^1(\T,\B(H))$ then $\A\in \Lrl2$.
\end{teor}
\begin{proof} (i) and (ii) is the content of  \cite[Corollary 4.13]{BB2}. We include the proof for completeness.
Since  $ {\widehat{( G_\A)}}_r(j-k)= T_{j-k} r^{j-k}$, the Toeplitz matrix associated to $(G_\A)_r$ turns out to be $\A_{(G_\A)_r}=P_r(\A)$.
Now,  Lemma \ref{cinfty} implies that $\|P_r(\A)\|_{\Bl2}=\|(G_\A)_r\|_{C(\T,\B(H))}$, which together with (\ref{equi00}) gives that $$\|\A\|_{\Bl2}= \sup_{0<r<1} \|P_r(\A)\|_{\Bl2}=\|G_\A\|_{H^\infty(\D, \B(H))}.$$

Since  $G_\A\in A(\D,\B(H))$ iff $(G_\A)_r\to G_\A$ in $H^\infty(\D, \B(H))$, we have that $G_\A\in A(\D,\B(H))$ if and only if $\A\in \Cl2$.

(iii) Using Theorem \ref{lema:norma} one has that $$\|P_r(\A)\|_{\Ml2}\le \|(G_\A)_r\|_{L^1(\T,\B(H))}$$
and due to (\ref{equi0}) we conclude that $\|\A\|_{\Ml2}\le \|G_\A\|_{H^1(\D,\B(H))}$.

(iv) follows from (iii) since $$\|P_r(\A)-\A\|_{\Ml2}\le \|(G_\A)_r-G_\A\|_{H^1(\D,\B(H))}$$ and taking limits as $r\to 1$.

\end{proof}

\vspace{.1in}
\[
\begin{tabular}{lccl}
Departamento de An\'alisis Matem\'{a}tico &   \\
Universidad de Valencia &  \\
46100 Burjassot &   \\
Valencia &   \\
Spain &  &  & \\
oscar.blasco@uv.es & Ismael.Garcia-Bayona@uv.es \\
\end{tabular}
\]

\end{document}